\documentclass[12pt]{amsart}
\usepackage{amssymb}
\usepackage{amsfonts}
\usepackage{amsmath}
\usepackage{amsthm}
\usepackage{graphicx}
\usepackage{epstopdf}
\usepackage{polynom}
\usepackage{bm}
\usepackage{cite}

\newtheorem{theorem}{Theorem}[section]

\newtheorem{corollary}[theorem]{Corollary}

\theoremstyle{definition}
\newtheorem{example}{Example}[section]

\newcommand{\rn}{\mathbb{R}}

\newcommand{\ds}{\displaystyle}

\newcommand{\C}{\mathbb{C}}

\newcommand{\inprod}[2]{\langle#1 | #2\rangle}

\begin{document}

\title{Densely Defined Multiplication on the Sobolev Space}
\author{Joel A. Rosenfeld}
\address{358 Little Hall, PO Box 118105
Gainesville FL 32611-8105}
\email{joelar@ufl.edu}
\date{June 7, 2013}

\begin{abstract}
Following Sarason's classification of the densely defined multiplication operators over the Hardy space, we classify the densely defined multipliers over the Sobolev space, $W^{1,2}[0,1]$.  In this paper we find that the collection of such multipliers for the Sobolev space is exactly the Sobolev space itself.  This sharpens a result of Shields concerning bounded multipliers.  The densely defined multiplication operators over the subspace $W_0 = \{ f \in W^{1,2}[0,1] : f(0)=f(1)=0 \}$ are also classified.  In this case the densely defined multiplication operators can be written as a ratio of functions in $W_0$ where the denominator is non-vanishing. This is proved using a contructive argument.
\end{abstract}

\maketitle

\section{Introduction}
Recall that a reproducing kernel Hilbert space (RKHS), $H$, over a set $X$ is a Hilbert space of functions $f: X \to \mathbb{C}$ for which the evaluation functionals $E_x f = f(x)$ are bounded. By the Riesz representation theorem, this tells us that for each $x \in X$ there is a function $k_x \in H$ such that $f(x) = \inprod{f}{k_x}$ for all $f \in H$.

Given a RKHS over a set $X$, $H$, a function $\phi: X \to H$ is a densely defined multiplier if the set $$D(M_\phi) = \{ f \in H : \phi f \in H \}$$ is dense in $H$.  The multiplication operator, $M_\phi f = \phi f$, is a closed operator \cite{franek}, so if $D(M_\phi) = H$, $M_\phi$ is a bounded operator by the closed graph theorem. Moreover, for all $x \in X$, the reproducing kernel $k_x$ is an eigenvector for the (densely defined) adjoint, $M_\phi^*$, with eigenvalue $\overline{\phi(x)}$ \cite{franek}.

Bounded multiplication operators are a well studied class of operators in Operator Theory. They provide straightforward examples for use in spectral theory, are viewed as transfer functions for linear systems, and in reproducing kernel Hilbert spaces, they interact nicely with the reproducing kernels. The classification of these operators is an important part of operator theory, and this classification has been carried out for the Hardy space \cite{aglermccarthy}, Fock space \cite{aglermccarthy}, Dirichlet space \cite{stegenga}, and the Bergman space \cite{luecking}.  For the classical Hardy space, $H^2$, the bounded multiplication operators are those operators with symbol $\phi \in H^\infty$, for example.

Bounded multipliers have also been classified in the case of the Sobolev Space, $$W^{1,2}[0,1] = \{ f: [0,1] \to \mathbb{C} : f \text{ is absolutely continuous }, f' \in L^2[0,1] \}$$ equipped with the inner product: $$\inprod{f}{g} = \int_0^1 f(x)\overline{g(x)} + f'(x)\overline{g'(x)} dx.$$ Alan Shields showed that the Sobolev space is a space of functions where the collection of multipliers is exactly the space itself \cite{halmos}.  This contrasts with the Hardy space, where the collection of multipliers is strictly contained within the Hardy space. Jim Agler in \cite{agler}, showed that like the multipliers for the Hardy space, the multiplication operators for the Sobolev space have the Nevanlinna-Pick property.

A thorough investigation of the bounded multipliers of Sobolev spaces on spaces of higher dimensions is carried out in \cite{mazya}. In the present work, densely defined multiplication operators are investigated over the RKHS $W^{1,2}[0,1]$. $W^{1,2}[0,1]$ is a special case of Sobolev spaces that has a reproducing kernel, since in general, elements of a Sobolev space are equivalence classes of functions that agree almost everywhere \cite{evans}.

The main result of this paper, Theorem \ref{allbdd}, states that Shields' result holds when you relax bounded to densely defined.  That is the collection of \emph{densely defined} multipliers over the Sobolev space is the space itself. In particular, the Sobolev space provides an example of a space where every densely defined multiplier is in fact bounded.

In contrast, the densely defined multipliers over the subspace $$W_0 = \{ f \in W^{1,2}[0,1] : f(0)=f(1)=0 \}$$ are not necessarily bounded operators, Theorem \ref{ddms}. It is demonstrated in Theorem \ref{zerofree} that every function $\phi$ that is a symbol for a densely defined multiplication operator over $W_0$ can be written as a ratio of functions $h, f \in W_0$ where $f$ does not vanish on $(0,1)$. This can be compared to Sarason's characterization of densely defined multipliers over the Hardy space $H^2$, where the densely defined multipliers are those functions in the Smirnov class, $N^{+} = \{ b/a : b, a \in H^\infty, a \text{ outer} \}$ \cite{sarason}. In particular densely defined multipliers over the Hardy space can be expressed as a ratio of functions in $H^2$ such that the denominator is nonvanishing.

\section{Densely Defined Multipliers for the Sobolev Space}

%
%
%
%

\begin{example}As we will see in the following theorem, the densely defined multipliers of $W_0$ are those functions that are well behaved everywhere but the endpoints of $[0,1]$.  Take for instance the topologist's sine curve $\phi(x) = \sin(1/x)$.


On any interval bounded away from zero, $\sin(1/x)$ is smooth.  To determine that $D(M_\phi)$ is dense, it is enough to recognize that the set of functions that vanish in a neighborhood of zero are in $D(M_\phi)$ and this collection of functions is dense in $W_0$.

We can apply the same reasoning to show that two other poorly behaved functions are symbols for densely defined multiplication operators: $\phi(x) = 1/x$ and $\exp(1/x)$.  Here $1/x$ is has a simple pole at $0$ and $e^{1/x}$ has an essential singularity.\end{example}

The following theorem can be compared to Theorem 2.3.2 in \cite{mazya} for bounded multiplication operators.

\begin{theorem}\label{ddms} A function $\phi: (0,1) \to \C$ is the symbol for a densely defined multiplier on $W_0$ iff $\phi \in W^{1,2}[a,b]$ for all $[a,b] \subset (0,1)$.\end{theorem}

\begin{proof} First suppose that $\phi$ is a densely defined multiplier on $W_0$.  For each $x_0 \in (0,1)$ there is a function $f \in D(M_\phi)$ such that $f(x_0) \neq 0$, this follows from the density of the domain.  If all functions in $f \in D(M_\phi)$ vanished at $x_0$, then $D(M_\phi) \subset \{ k_{x_0}\}^\perp$.  This would mean that $\overline{D(M_\phi)} = \left(D(M_\phi)^\perp\right)^\perp \subset \{k_{x_0}\}^\perp$, and this contradicts the density of the domain.  Let $h = M_\phi f$, so that $\phi(x) = h(x)/f(x)$ in a neighborhood of $x_0$.  The functions $h$ and $f$ are differentiable almost everywhere in a neighborhood of $x_0$, so then is $\phi$.  Since $x_0$ is arbitrary, $\phi$ is differentiable almost everywhere on $(0,1)$.

Fix $[a,b] \subset (0,1)$.  By way of compactness, there exists a finite collection of functions $\{f_1, f_2, ..., f_k\} \subset D(M_g)$ together with subsets $$[a,t_1), \ (s_2,t_2), \ (s_3,t_3), \ ... , \ (s_{k-1}, t_{k-1}), \ (s_k, b]$$ so that the subsets cover $[a,b]$ and $f_i$ is bounded away from zero on $[s_i,t_i]$, here we take $s_1 = a$ and $t_k = b$.

Since $f_i$ is bounded away from zero on $[s_i, t_i]$, $\phi$ is absolutely continuous on each $[s_i, t_i]$ and hence on $[a,b]$.  We wish to show that $\phi \in W^{1,2}[a,b]$ so we set out to show $\phi$ and $\phi^\prime$ are in $L^2[a,b]$.  Set $h_i = \phi f_i$ and by the product rule we find $h_i^\prime = \phi^\prime f_i + f_i^\prime \phi$ almost everywhere.

Since $\phi$ is continuous on $[s_i, t_i]$, we have $\phi \in L^2[s_i,t_i]$.  The function $f^\prime$ is also in  $L^2[s_i,t_i]$ which implies $\phi f^\prime \in L^2[s_i,t_i]$, because $\phi$ is bounded.  Therefore $h^\prime_i-\phi f^\prime_i = \phi^\prime f_i \in L^2[s_i,t_i]$.  By construction, $f_i$ does not vanish on $[s_i,t_i]$, so $$\inf_{[s_i, t_i]} |f_i(x)|^2\int_0^1 |\phi^\prime|^2 dx \le \int_0^1 |\phi^\prime f_i|^2 dx < \infty.$$  Thus $\phi^\prime \in L^2[s_i,t_i]$ for each $i=1,2,...,k$, and $\phi \in W^{1,2}[a,b]$.

For the other direction, suppose that $\phi \in W^{1,2}[a,b]$ for all $[a,b] \subset (0,1)$.  Let $f \in W_0$ such that $f$ has compact support in $(0,1)$.  Let $[a,b]$ be a compact subset of $(0,1)$ containing the support of $f$.  Outside of $[a,b]$, $f$ is identically zero and so $f^\prime \equiv 0$ as well.

The function $\phi f$ is in $L^2[a,b]$ since it is continuous.   Also the function $\phi f^\prime \in L^2[a,b]$ since $\phi$ is continuous and $f^\prime \in L^2[a,b]$, and $\phi^\prime f \in L^2[a,b]$ for the opposite reason.  Thus $h := \phi f \in W^{1,2}[a,b]$, and since it vanishes outside the interval, $h \in W_0$.  Therefore $f \in D(M_\phi)$, and compactly supported functions are dense in $W_0$.  Thus $\phi$ is a densely defined multiplication operator.\end{proof}

This proof can be extended to prove the following:

\begin{theorem}\label{allbdd}For the Sobolev space, $W^{1,2}[0,1]$, the collection of symbols of densely defined multipliers is $W^{1,2}[0,1]$.  In particular, all the densely defined multipliers are bounded.\end{theorem}

\begin{proof}Taking advantage of compactness, we can find a finite collection of functions $f_1, f_2, ... , f_k \in D(M_\phi)$ and intervals $$[0,t_1), \ (s_2,t_2), \ (s_3,t_3), \ ... , \ (s_{k-1}, t_{k-1}), \ (s_k, 1]$$ that cover $[0,1]$ for which $f_i$ is bounded away from zero on $[s_i, t_i]$.  (Taking $s_1 = 0$ and $t_k = 1$.)

This time we are allowed to let $s_1=0$ and $t_k=1$, since the Sobolev space does not require that functions vanish at the endpoints.  This means we can find two functions that do not vanish at $0$ and $1$ respectively.

Running the same argument as in the previous theorem we can see that $\phi \in W^{1,2}[0,1]$.  This means that $D(M_\phi) = W^{1,2}[0,1]$ and $M_\phi$ is a bounded multiplication operator.
\end{proof}
For the Hardy space there are many more densely defined multipliers than bounded ones \cite{sarason}.  In fact the Hardy space is properly contained inside of its collection of densely defined multipliers, the Smirnov class.  In the Sobolev space we see that the collection of densely defined multipliers is exactly the Sobolev space itself, and they are all bounded. The same methods can be used to show the following corollary:

\begin{corollary}Given the Sobolev space $W^{1,2}(\rn)$, a function $\phi$ is a densely defined multiplier for $W^{1,2}(\rn)$ iff $\phi \in W^{1,2}(E)$ for all compact intervals $E$ of $\rn$.\end{corollary}
\section{Local to Global Non-Vanishing Denominator}
We saw in Theorem \ref{ddms} that for any point $x \in (0,1)$ we can find a function in the domain that does not vanish in a neighborhood of that point.  In other words, we used a local non-vanishing property.  Now that we have an explicit description of the densely defined multipliers of $W_0$, we can sharpen this to finding a globally nonvanishing function inside of the domain.  This means that a densely defined multiplication operator, $\phi$, on $W_0$ can be expressed as a ratio of two functions in $W_0$ where the denominator is non-vanishing.

Ideally given any densely defined multiplication operator over a Hilbert function space $H$, we would like to express its symbol as a ratio of two functions from $H$ such that the denominator is non-vanishing.  As we saw in the proofs of Theorems \ref{ddms} and \ref{allbdd}, we can always do this locally.  In \cite{sarason}, this was achieved for the Hardy space through an application of the inner-outer factorization, but there is no such factorization theorem for functions in the Sobolev space.  This means we need to try something a little different.

Looking at our three poorly behaved functions we can rewrite them as quotients of functions in $W_0$ as follows: $$\begin{array}{ccc}
\sin(1/x)&=&\ds \frac{x^2(1-x)\sin(1/x)}{x^2(1-x)}\vspace{.2in}\\
\ds1/x&=&\ds \frac{x(1-x)}{x^2(1-x)}\vspace{.2in}\\
\exp(1/x)&=&\ds \frac{x(1-x)}{x(1-x)\exp(-1/x)}\end{array}$$

In the following theorem, a constructive method is described that finds such a ratio of $W_0$ functions.

\begin{theorem}\label{zerofree} If $\phi$ is a densely defined multiplier for $W_0$, then there exists $f \in D(M_\phi)$ such that $f(x) \neq 0$ on $(0,1)$.\end{theorem}

\begin{proof}First if $\phi \in W^{1,2}[0,1]$, we are finished trivially by writing $\phi(x)= \frac{x(1-x)g(x)}{x(1-x)}$.  We will assume that $\phi \not\in W^{1,2}[0,1]$.  First by Theorem \ref{ddms} $\phi$ is locally absolutely continuous on $(0,1)$, so if $\phi \not \in W^{1,2}[0,1]$ then either $\phi \not \in W^{1,2}[0,1/2]$ and/or $\phi \not \in W^{1,2}[1/2,1]$.  Assume without loss of generality that $\phi \not \in W^{1,2}[0,1/2]$.  The function $\phi$ is absolutely continuous on $(0,1/2)$ but not in $W^{1,2}[0,1/2]$ which means either $\phi$ or $\phi^\prime$ is not in $L^2[0,1/2]$. We know by Theorem \ref{ddms} that $\phi \in W^{1,2}[a,\frac12]$ for each $a > 0$, but $\phi \not\in W^{1,2}[0,1]$.  

Both of the functions $\phi$ and $\phi^\prime$ are in $L^2[a, \frac12]$ for all $a > 0$: $\int_a^{.5} |\phi|^2 + |\phi^\prime|^2 dx < \infty$. Construct the increasing sequence:
$$a_n = \int_{\frac{1}{2^{n+1}}}^{\frac{1}{2^n}} |\phi|^2 + |\phi^\prime|^2 dx.$$
Define $b_n = \min\left\{ (a_n)^{-1}, (a_{n-1})^{-1}, 1\right\}$.  Notice that $a_n b_{n+1}$, $a_n b_n$ and $b_n \le 1$ for all $n$.

Now we can begin constructing our non-vanishing function $f$. Let $f$ be the function that linearly interpolates the points $\left\{(\frac{1}{2^n}, \frac{b_n}{2^{2n}})\right\}_{n=1}^{\infty}$.  Also define $f(0)=0$, and note that $\lim_{x\to0^+} f(x) = 0$.  To be more precise we define auxiliary functions $L_n(x)$ by:
$$L_n(x) = \left\{ \begin{array}{lcl}\frac{4(b_n) - (b_{n+1})}{2^{n+1}}(x-2^{-n}) + (2^{-2n}(b_n)) &:& x \in (2^{-(n+1)}, 2^{-n}]\\
0 &:& \text{otherwise} \end{array}\right.$$


Now $f$ can be written as $f = \sum_{n=1}^{\infty} L_n(x)$.  The function $f$ is continuous on $[0,\frac12]$ and differentiable almost everywhere.  By using calculations with $L_n$ it is straightforward to show $f, f^\prime \in L^2[0,\frac12]$, since the slopes of each of these functions was chosen to decrease geometrically.  Thus $f \in W^{1,2}[0,\frac12]$ and $f(0) = 0$.

The function $\phi f$ is continuous on $(0,1/2)$ and differentiable almost everywhere.  We wish to show that both $\phi f$ and $(\phi f)^\prime = \phi^\prime f + f^\prime \phi$ are in $L^2[0,\frac12]$.  First we have:
$$\begin{array}{rcl}\ds\int_0^{0.5}|\phi f|^2dx &=&\ds \sum_{n=1}^{\infty}\int_{1/2^{n+1}}^{1/2^n}|\phi L_n|^2dx\\ &\le&\ds \sum_{n=1}^{\infty} a_n \max\left\{\left(\frac{b_n}{2^{2n}}\right)^2, \left(\frac{b_{n+1}}{2^{2(n+1)}}\right)^2\right\}<\infty.\end{array}$$
Similarly
$$\begin{array}{rcl}\ds\int_0^{0.5}|\phi^\prime f|^2dx &=&\ds \sum_{n=1}^{\infty}\int_{1/2^{n+1}}^{1/2^n}|\phi^\prime L_n|^2dx\\ &\le&\ds \sum_{n=1}^{\infty} a_n \max\left\{\left(\frac{b_n}{2^{2n}}\right)^2, \left(\frac{b_{n+1}}{2^{2(n+1)}}\right)^2\right\}<\infty,\end{array}$$
and
$$\begin{array}{rcl}\ds\int_0^{0.5}|\phi f^\prime|^2dx &=&\ds \sum_{n=1}^{\infty}\int_{1/2^{n+1}}^{1/2^n}|\phi L_n^\prime|^2dx\\ &\le&\ds \sum_{n=1}^\infty a_n \left( \frac{4(b_n) - (b_{n+1})}{2^{n+1}} \right)^2 < \infty.\end{array}$$
Here we see each integral is dominated by a geometric series, and so $\phi f, (\phi f)^\prime \in L^2[0,1/2]$.  Next show that $\phi f$ is absolutely continuous on $[0,1]$. Notice that $\phi f$ and $(\phi f)^\prime$ are in $L^2[0,1/2]$, and $\phi f$ is absolutely continuous on $[a,1/2]$ for any $0<a<1/2$.  For the moment call $h = \phi f$.  
We will show that $xh(x)$ is in $W_0$. On every interval $(a,1/2]$, $xh(x)$ is absolutely continuous, which is equivalent to $$ah(a) = \frac12 h\left( \frac12 \right) + \int_{a}^{\frac12} \frac{d}{dt} (t h(t) ) dt$$ for all $0<a<0.5$. Showing that this holds for $a=0$, demonstrates that $xh(x)$ is absolutely continuous on $[0,1]$.  First note that
$$\begin{array}{rcl}\ds \lim_{x\to0^+} |x h(x)| &=&\ds \lim_{x\to0^+} \left|x h\left( \frac12 \right) + x \int_{\frac12}^{x} h^\prime(t) dt \right| \vspace{.05in}\\ &\le&\ds \lim_{x\to0^+} \left( |x| \left|h\left( \frac12 \right)\right| + |x| \int_x^{\frac12} |h^\prime(t)|dt\right).\end{array}$$
The last limit is zero, since $h^\prime \in L^2[0,1/2] \subset L^1[0,1/2]$.  Thus $xh(x)$ can be defined as a continuous function on $[0,1]$ by declaring this function as zero at $x=0$.  Finally, $$\begin{array}{rcl}\ds \frac12 h\left(\frac12\right) + \int_{0}^{1/2} \frac{d}{dt} \left( th(t) \right) dt &=&\ds \lim_{a \to 0^+} \left(\frac12 h\left(\frac12\right) + \int_{a}^{1/2} \frac{d}{dt} \left( th(t) \right) dt\right) \vspace{.1in}\\
&=&\ds \lim_{a \to 0^+} \left(\frac12 h \left( \frac12 \right) + \left( ah(a) - \frac12 h\left( \frac12 \right) \right)\right) \vspace{.1in}\\
&=&\ds \lim_{a \to 0^+} ah(a) = 0\end{array}.$$

This proves that $xh(x)$ is absolutely continuous on $[0,1/2]$.  Therefore, the function $xh$ is in $W^{1,2}[0,1/2]$ and $(xh)(0)=0$.  Thus if we construct $f$ as above for both the left half of $[0,1]$ and the right half, then the function $x(1-x) f$ does not vanish on $(0,1)$ and is in $D(M_\phi) \subset W_0$.\end{proof}

\section{Remarks}

We leave with one last note concerning densely defined multipliers on the Sobolev space.  We know that if a multiplier is bounded, then its symbol is bounded by the norm of the operator.  The question arrises, if $g$ is known to be a densely defined multiplication operator over $W_0$ and $\sup_{x\in (0,1)} |\phi(x)| < \infty$ is $M_\phi$ a bounded multiplier?

The answer is: not necessarily.  We can produce a counterexample by examining $$\phi(x) = \sqrt{1/4-(x-1/2)^2}$$ which is bounded on $[0,1]$ by 1/2.  By the work above, we know that $M_\phi$ is a densely defined multiplier, since $\phi \in W^{1,2}[a,b]$ for every $[a,b] \subset (0,1)$.  However, since $\phi^\prime$ is not bounded on $[0,1]$, $M_\phi 1 = \phi\cdot 1 \not \in W_0$.  Therefore even though $\phi$ is a bounded function, the multiplier $M_\phi$ is not bounded.

\bibliographystyle{plain}
\bibliography{sobolev}
\end{document}